\documentclass[11pt]{amsart}

\theoremstyle{plain}
\newtheorem{theorem}                 {Theorem}      [section]

\newtheorem{conjecture}   [theorem]  {Conjecture}

\newtheorem*{theorem*}               {Theorem\;{\bf \ref{th:
bih surf s4_parallel H}}}
\newtheorem*{probl}                  {Open problem}

\theoremstyle{definition}

\numberwithin{equation}{section}

\def \1{\mbox{${\mathbf 1}$}}

\def \s{\mbox{${\mathbb S}$}}

\def \rrm{\mbox{${\scriptstyle \frac{1}{\sqrt 2}}$}}

\DeclareMathOperator{\trace}{trace}
\DeclareMathOperator{\grad}{grad}

\DeclareMathOperator{\id}{Id}

\begin{document}

\title{Biharmonic surfaces of $\s^4$}
\author{A. Balmu\c s}
\author{C. Oniciuc}
\address{Faculty of Mathematics, ``Al.I.~Cuza'' University of Iasi\\
\newline
Bd. Carol I Nr. 11 \\
700506 Iasi, ROMANIA}
\email{adina.balmus@uaic.ro,
oniciucc@uaic.ro}

\dedicatory{}


\thanks{}

\begin{abstract}
In this note we prove that a constant mean curvature surface is
proper-biharmonic in the unit Euclidean sphere $\s^4$ if and only if
it is minimal in a hypersphere $\s^3(\rrm)$.
\end{abstract}


\maketitle

\section{Introduction}

{\it Biharmonic maps} $\varphi:(M,g)\to(N,h)$ between Riemannian
manifolds are critical points of the {\em bienergy} functional
$$
E_2(\varphi)=\frac{1}{2}\int_{M}\,|\tau(\varphi)|^2\,v_g,
$$
where $\tau(\varphi)=\trace\nabla d\varphi$ is the {tension field}
of $\varphi$ that vanishes for harmonic maps (see~\cite{JEJHS}). The
Euler-Lagrange equation corresponding to $E_2$ is given by the
vanishing of the {\it bitension field}
$$
\tau_2(\varphi)=-J^{\varphi}(\tau(\varphi))=-\Delta\tau(\varphi)
-\trace R^N(d\varphi,\tau(\varphi))d\varphi,
$$
where $J^{\varphi}$ is formally the Jacobi operator of $\varphi$
(see \cite{GYJ1}). The operator $J^{\varphi}$ is linear, thus any
harmonic map is biharmonic. We call {\it proper-biharmonic} the
non-harmonic biharmonic maps.

The study of {\it proper-biharmonic submanifolds}, i.e. submanifolds
such that the inclusion map is non-harmonic (non-minimal)
biharmonic, constitutes an important research direction in the
theory of biharmonic maps. The first ambient spaces taken under
consideration for their proper-biharmonic submanifolds were the
spaces of constant sectional curvature. Non-existence results were
obtained for proper-biharmonic submanifolds in Euclidean and
hyperbolic spaces (see \cite{BMO1, RCSMCO2, BYC, ID, THTV}).

The case of the Euclidean sphere is different. Indeed, the
hypersphere $\s^{n-1}(\rrm)$ and the generalized Clifford torus
$\s^{n_1}(\rrm)\times\s^{n_2}(\rrm)$, $n_1+n_2=n-1$, $n_1\neq n_2$,
are the main examples of proper-biharmonic submanifolds in $\s^n$
(see \cite{RCSMCO1, GYJ1}). Moreover, the following
\begin{conjecture}[\cite{BMO1}]\label{conj: hypersurf_sphere}
The only proper-biharmonic hypersurfaces in $\s^{n}$ are the open
parts of hyperspheres $\s^{n-1}(\rrm)$ or of generalized Clifford
tori $\s^{n_1}(\rrm)\times \s^{n_2}(\rrm)$, $n_1+n_2=n-1$, $n_1\neq
n_2$.
\end{conjecture}
was proposed. This proved to be true for certain classes of
hypersurfaces with additional geometric properties (see \cite{BMO1,
BMO2}).

In codimension greater than $1$, the family of proper-biharmonic
submanifolds is rather large. For example, any minimal submanifold
in $\s^{n-1}(\rrm)$ is proper-biharmonic in $\s^n$. In particular,
any minimal surface in $\s^3(\rrm)$, see \cite{HBL}, provides a
proper-biharmonic surface in $\s^4$.

All proper-biharmonic submanifolds of $\s^2$ and $\s^3$ were
determined (see \cite{CMP, RCSMCO1}). The next step towards the
classification of proper-biharmonic submanifolds in spheres is
represented by the case of $\s^4$, and the first achievement was the
proof of Conjecture~\ref{conj: hypersurf_sphere} for compact
hypersurfaces in $\s^4$ (see \cite{BMO2}). Since all
proper-biharmonic curves in $\s^n$, and therefore in $\s^4$, were
determined (see \cite{RCSMCO2}), the aim of this paper is to give a
partial answer to the
\begin{probl}[\cite{BMO2}] Are there other proper-biharmonic surfaces in $\s^4$,
apart from the minimal surfaces of $\s^3(\rrm)$?
\end{probl}
We show that the answer is negative in the case of proper-biharmonic
surfaces with constant mean curvature in $\s^4$ (Theorem \ref{th:
bih surf s4_parallel H}).

For other results on proper-biharmonic submanifolds in spaces of
non-constant sectional curvature see, for example, \cite{FLMO,
TIJIHU, MO, YO}.

\section{Preliminaries}

Let $\varphi:M\to\s^n$  be the canonical inclusion of a
submanifold $M$ in the $n$-dimensional unit Euclidean sphere. The
expressions assumed by the tension and bitension fields are
\begin{equation}\label{eq: tens+bitens E^n(c)}
\tau(\varphi)=mH,\qquad\qquad \tau_2(\varphi)=-m(\Delta H-mH),
\end{equation}
where $H$ denotes the mean curvature vector field of $M$ in
$\s^n$, while $\Delta$ is the rough Laplacian on
$\varphi^{-1}T\s^n$.

The following characterization result proved to be the main
ingredient in the study of proper-biharmonic submanifolds in
spheres.

\begin{theorem}[\cite{BYC2, CO1}]\label{th: bih subm E^n(c)}
The canonical inclusion $\varphi:M^m\to\s^n$ of a submanifold $M$
in the $n$-dimensional unit Euclidean sphere $\s^n$ is biharmonic
if and only if
\begin{equation}\label{eq: caract_bih_spheres}
\left\{
\begin{array}{l}
\ \Delta^\perp H+\trace B(\cdot,A_H\cdot)-mH=0,
\\ \mbox{} \\
\ 4\trace A_{\nabla^\perp_{(\cdot)}H}(\cdot)+m\grad(\vert H
\vert^2)=0,
\end{array}
\right.
\end{equation}
where $A$ denotes the Weingarten operator, $B$ the second
fundamental form, $H$ the mean curvature vector field,
$\nabla^\perp$ and $\Delta^\perp$ the connection and the Laplacian
in the normal bundle of $M$ in $\s^n$.
\end{theorem}

Using the main examples, two methods of construction for
proper-biharmonic submanifolds in spheres were given.

\begin{theorem}[Composition property, \cite{RCSMCO2}] \label{th: rm_minim}
Let $M$ be a minimal submanifold of\, $\s^{n-1}(a)\subset\s^n$.
Then $M$ is proper-biharmonic in $\s^{n}$ if and only if $a=\rrm$.
\end{theorem}

We note that such submanifolds are pseudo-umbilical, i.e.
$A_H=|H|^2\id$, have parallel mean curvature vector field and mean
curvature $|H|=1$.

\begin{theorem}[Product composition property, \cite{RCSMCO2}]\label{th:hipertor}
Let $M_1^{m_1}$ and $M_2^{m_2}$ be two mini\-mal submanifolds of
$\s^{n_1}(r_1)$ and $\s^{n_2}(r_2)$, respectively, where
$n_1+n_2=n-1$, $r_1^2+r_2^2=1$. Then $M_1\times M_2$ is
proper-biharmonic in $\s^n$ if and only if $r_1=r_2=\rrm$ and
$m_1\neq m_2$.
\end{theorem}
The proper-biharmonic submanifolds obtained in this way are no
longer pseudo-umbilical, but still have parallel mean curvature
vector field and their mean curvature is bounded, $|H|\in(0,1)$.
For dimension reasons, this second method cannot be applied in
order to produce proper-biharmonic surfaces in $\s^4$.

In \cite{AEMS, S1} the authors obtained explicit examples of
proper-biharmonic submanifolds in $\s^5$ with constant mean
curvature, which are neither pseudo-umbilical nor of parallel mean
curvature vector field.

We note that all known examples of proper-biharmonic submanifolds in
$\s^n$ have constant mean curvature.

We end by recalling here the following results which are needed in
the next section.

\begin{theorem}[\cite{BMO1}]\label{th: classif_pseudo_umb_codim2}
Let $M^m$ be a pseudo-umbilical submanifold in $\s^{m+2}$, $m\neq
4$. Then $M$ is proper-biharmonic in $\s^{m+2}$ if and only if it
is minimal in $\s^{m+1}(\rrm)$.
\end{theorem}

\begin{theorem}[\cite{BMO1}]\label{th: classif_surf_parallelH}
Let $M^2$ be a surface with parallel mean curvature vector field
in $\s^n$. Then $M$ is proper-biharmonic in $\s^n$ if and only if
it is minimal in $\s^{n-1}(\rrm)$.
\end{theorem}

\section{Biharmonic surfaces with constant mean curvature in $\s^4$}

We shall prove the following

\begin{theorem}\label{th: bih surf s4_parallel H}
Let $M^2$ be a  proper-biharmonic constant mean curvature surface
in $\s^4$. Then $M^2$ is minimal in $\s^3(\rrm)$.
\end{theorem}

\begin{proof}

Following \cite{BYC_ISH}, we shall first prove that any
proper-biharmonic constant mean curvature surface in $\s^4$ has
parallel mean curvature vector field. Then we shall conclude by
using Theorem \ref{th: classif_surf_parallelH}.

Denote by $H$ the mean curvature vector field of $M^2$ in $\s^4$.
Since $M$ is proper-biharmonic with constant mean curvature, its
mean curvature does not vanish at any point and we denote by
\begin{equation}\label{eq: E_3}
E_3=\displaystyle{\frac{H}{|H|}}\in C(NM).
\end{equation}
Consider $\{E_1,E_2\}$ to be a local orthonormal frame field on $M$
around an arbitrary fixed point $p\in M$ and let $E_4$ be a local
unit section in the normal bundle, orthogonal to $E_3$. We can
assume that $\{E_1,E_2,E_3, E_4\}$ is the restriction of a local
orthonormal frame field around $p$ on $\s^4$, also denoted by
$\{E_1,E_2,E_3,E_4\}$.

Denote by $B$ the second fundamental form of $M$ in $\s^4$ and by
$A_3$ and $A_4$ the Weingarten operators associated to $E_3$ and
$E_4$, respectively.

Let $\nabla^{\mathbb{S}^4}$ and $\nabla$ be the Levi-Civita
connections on $\s^4$ and on $M$, respectively, and denote by
$\omega_A^B$ the connection $1$-forms of $\s^4$ with respect to
$\{E_1,E_2,E_3,E_4\}$, i.e.
\begin{equation}\label{eq: conn_forms}
\nabla^{\mathbb{S}^4} E_A=\omega_A^B E_B,\qquad A,B=1,\ldots,4.
\end{equation}

From \eqref{eq: E_3} we have $H=|H|E_3$ and, since
$2H=B(E_1,E_1)+B(E_2,E_2)$, we obtain that
\begin{eqnarray}\label{eq: trace A4}
0&=&2\langle H,E_4\rangle=\langle B(E_1,E_1),E_4\rangle+\langle
B(E_2,E_2),E_4\rangle\nonumber\\
&=&\langle A_4(E_1),E_1\rangle+\langle A_4(E_2),E_2\rangle,
\end{eqnarray}
i.e. $\trace A_4=0$. As a consequence, we have
\begin{eqnarray}\label{eq: norm A4}
|A_4|^2&=&|A_4(E_1)|^2+|A_4(E_2)|^2\nonumber\\
&=&\langle A_4(E_1),E_1\rangle^2+2\langle A_4(E_1),
E_2\rangle^2+\langle A_4(E_2),E_2 \rangle^2\nonumber\\
&=& 2 \big(\langle A_4(E_1),E_1\rangle^2+\langle A_4(E_1),
E_2\rangle^2\big).
\end{eqnarray}

The tangent part of the biharmonic equation \eqref{eq:
caract_bih_spheres} now writes
\begin{equation}\label{eq: caract_bih_M^2}
A_{\nabla^\perp_{E_1}E_3}(E_1)+A_{\nabla^\perp_{E_2}E_3}(E_2)=0.
\end{equation}
Since
\begin{eqnarray*}
\nabla^\perp_{E_1}E_3 &=& \langle \nabla^\perp_{E_1}E_3, E_3\rangle
E_3 + \langle \nabla^\perp_{E_1}E_3, E_4\rangle E_4=\langle
\nabla^{\mathbb{S}^4}_{E_1}E_3, E_4\rangle
E_4\\&=&\omega^4_3(E_1)E_4,
\end{eqnarray*}
and
$$
\nabla^\perp_{E_2}E_3 = \omega^4_3(E_2)E_4,
$$
from \eqref{eq: caract_bih_M^2} we get
\begin{equation}\label{eq: omega_A4}
\omega^4_3(E_1)A_4(E_1)+\omega^4_3(E_2)A_4(E_2)=0.
\end{equation}
Considering now the scalar product by $E_1$ and $E_2$ in
\eqref{eq: omega_A4}, we obtain
\begin{equation}\label{eq: omega_A4_bis}
\left\{
\begin{array}{l}
\ \langle A_4(E_1), E_1\rangle \omega^4_3(E_1) +\langle A_4(E_2),
E_1\rangle \omega^4_3(E_2)=0,
\\ \mbox{} \\
\ \ \langle A_4(E_1), E_2\rangle \omega^4_3(E_1) +\langle
A_4(E_2), E_2\rangle \omega^4_3(E_2)=0.
\end{array}
\right.
\end{equation}
Equations \eqref{eq: omega_A4_bis} can be thought of as a linear
homogeneous system in $\omega^4_3(E_1)$ and $\omega^4_3(E_2)$. By
using \eqref{eq: trace A4} and \eqref{eq: norm A4}, the determinant
of this system is equal to $\displaystyle{-\frac{1}{2}|A_4|^2}$.

Suppose now that $(\nabla^\perp H)(p)\neq 0$. Then there exists a
neighborhood $U$ of $p$ in $M$ such that $\nabla^\perp H\neq 0$,
at any point of $U$. Since
$$
\nabla^\perp H=|H|\nabla^\perp
E_3=|H|\{\omega^4_3(E_1)E_1^\flat\otimes
E_4+\omega^4_3(E_2)E_2^\flat\otimes E_4\},
$$
the hypothesis $\nabla^\perp H\neq 0$ on $U$ implies that
\eqref{eq: omega_A4_bis} admits non-trivial solutions at any point
of $U$. Therefore, the determinant of \eqref{eq: omega_A4_bis} is
zero, which means that $|A_4|^2=0$, i.e. $A_4=0$ on $U$.

We have two cases.

\noindent {\bf Case I.} If $U$ is pseudo-umbilical in $\s^4$, i.e.
$A_3=|H|\id$, from Theorem \ref{th: classif_pseudo_umb_codim2} we
get that $U$ is minimal in $\s^3(\rrm)$ and we have a
contradiction, since any minimal surface in $\s^3(\rrm)$ has
parallel mean curvature vector field in $\s^4$.

\noindent {\bf Case II.} Suppose that there exists $q\in U$ such
that $A_3(q)\neq |H|\id$. Then, eventually by restricting $U$, we
can suppose that $A_3\neq |H|\id$ on $U$. Since the principal
curvatures of $A_3$ have constant multiplicity $1$, we can suppose
that $E_1$ and $E_2$ are such that
$$
A_3(E_1)=k_1E_1, \qquad A_3(E_2)=k_2E_2,
$$
where $k_1\neq k_2$ at any point of $U$. As $A_4=0$, we obtain
\begin{equation}\label{eq: expr B}
B(E_1,E_1)=k_1E_3,\quad B(E_1,E_2)=0,\quad B(E_2,E_2)=k_2E_3,
\end{equation}
on $U$.

\noindent In the following we shall use the Codazzi and Gauss
equations in order to get to a contradiction.

\noindent The Codazzi equation is given in this setting by
\begin{equation}\label{eq: Codazzi}
0=(\nabla^{\mathbb{S}^4}_X B)(Y,Z,\eta)-(\nabla^{\mathbb{S}^4}_Y
B)(X,Z,\eta), \quad \forall X, Y, Z\in C(TM), \forall \eta\in
C(NM),
\end{equation}
where $\nabla^{\mathbb{S}^4}_X B$ is defined by
\begin{eqnarray*}
(\nabla^{\mathbb{S}^4}_X B)(Y,Z,\eta)&=&X\langle
B(Y,Z),\eta\rangle -\langle B(\nabla_X Y,Z),\eta\rangle-\langle
B(Y,\nabla_X Z),\eta\rangle\\&&-\langle
B(Y,Z),\nabla^\perp_X\eta\rangle.
\end{eqnarray*}
For $X=Z=E_1$, $Y=E_2$ and $\eta=E_3$, equation \eqref{eq:
Codazzi} leads to
\begin{eqnarray}\label{eq: Codazzi1}
0&=&E_1\langle B(E_2,E_1),E_3\rangle-E_2\langle
B(E_1,E_1),E_3\rangle\nonumber\\&&-\langle B(\nabla_{E_1}E_2, E_1),
E_3\rangle+\langle B(\nabla_{E_2}E_1,
E_1),E_3\rangle\nonumber\\&&-\langle B(E_2,\nabla_{E_1}E_1),
E_3\rangle+\langle B(E_1,\nabla_{E_2}E_1), E_3)\nonumber\\&&-
\langle B(E_2,E_1),\nabla^\perp_{E_1}E_3\rangle +\langle
B(E_1,E_1),\nabla^\perp_{E_2}E_3\rangle.
\end{eqnarray}
Now, from \eqref{eq: expr B} we have
$$
B(\nabla_{E_1}E_2, E_1) = k_1\omega^1_2(E_1) E_3, \qquad
B(E_2,\nabla_{E_1}E_1) = -k_2\omega^1_2(E_1) E_3,
$$
$$
B(\nabla_{E_2}E_1, E_1) = 0, \qquad \langle
B(E_1,E_1),\nabla^\perp_{E_2}E_3\rangle=0,
$$
thus \eqref{eq: Codazzi1} implies
\begin{equation}\label{eq: Codazzi_cons1}
E_2(k_1)=(k_2-k_1)\omega^1_2(E_1).
\end{equation}
Analogously, for $X=Z=E_2$, $Y=E_1$ and $\eta=E_3$ in \eqref{eq:
Codazzi}, we obtain
\begin{equation}\label{eq: Codazzi_cons2}
E_1(k_2)=(k_2-k_1)\omega^1_2(E_2).
\end{equation}
For $X=Z=E_1$, $Y=E_2$ and $\eta=E_4$ in \eqref{eq: Codazzi}, we
obtain
\begin{eqnarray*}
0&=&\langle B(E_2,E_1),\nabla^\perp_{E_1}E_4\rangle-\langle
B(E_1,E_1),\nabla^\perp_{E_2}E_4\rangle\\
&=&-k_1\langle E_3,\nabla^\perp_{E_2}E_4\rangle,
\end{eqnarray*}
which implies
\begin{equation}\label{eq: Codazzi_cons3}
k_1\omega^4_3(E_2)=0.
\end{equation}
Analogously, for $X=Z=E_2$, $Y=E_1$ and $\eta=E_4$ in \eqref{eq:
Codazzi}, we obtain
\begin{equation}\label{eq: Codazzi_cons4}
k_2\omega^4_3(E_1)=0.
\end{equation}
Since $\nabla^\perp H\neq 0$ on $U$, we can suppose that
$\omega^4_3(E_1)\neq 0$ on $U$. This, together with \eqref{eq:
Codazzi_cons4}, leads to $k_2=0$. From here we get $|k_1|=2|H|\neq
0$, and consequently $k_1$ is a non-zero constant. As $k_1\neq
k_2$, from \eqref{eq: Codazzi_cons1} and \eqref{eq: Codazzi_cons2}
we obtain
\begin{equation}\label{eq: M flat}
\omega^1_2(E_1)=\omega^1_2(E_2)=0,
\end{equation}
thus $M$ is flat.

\noindent Consider now the Gauss equation,
\begin{eqnarray}\label{eq: Gauss_EQ}
\langle R^{\mathbb{S}^4}(X,Y)Z,W\rangle&=&\langle
R(X,Y)Z,W\rangle\nonumber\\&&+\langle B(X,Z),B(Y,W)\rangle-\langle
B(X,W),B(Y,Z)\rangle.
\end{eqnarray}
As $M$ is flat, for $X=W=E_1$ and $Y=Z=E_2$, equations \eqref{eq:
Gauss_EQ} and \eqref{eq: expr B} lead to
\begin{eqnarray}
1&=&\langle B(E_1,E_2),B(E_2,E_1)\rangle-\langle
B(E_1,E_1),B(E_2,E_2)\rangle=-k_1 k_2\nonumber\\
&=&0,
\end{eqnarray}
and we have a contradiction.

Therefore, $\nabla^\perp H=0$ and we conclude.
\end{proof}

\end{document}